\documentclass{article}
\usepackage{amsmath, amsfonts, newlfont, latexsym, amscd, amsxtra, amssymb, amsbsy, amsgen, amsopn, fontenc, fontsmpl, amsthm, xypic, amsthm}
\usepackage[left=3.5truecm,nohead,right=3.5truecm,top=3.5truecm,bottom=3.5truecm]{geometry}
\newtheorem{definition}{Definition}[section]
\newtheorem{lemma}[definition]{Lemma}

\newtheorem{example}[definition]{Example}
\newtheorem{proposition}[definition]{Proposition}
\newtheorem{corollary}[definition]{Corollary}
\newtheorem{theorem}[definition]{Theorem}

\def\q{\frak q}
\def\reg{\operatorname{reg}}

\begin{document}

\makeatletter      
\renewcommand{\ps@plain}{%
    \renewcommand{\@oddhead}{\textrm{REGULARITY INDEX, HILBERT FUNCTION AND EMBEDDING}\hfil\textrm{\thepage}}%
     \renewcommand{\@evenhead}{\@oddhead}%
     \renewcommand{\@oddfoot}{}
     \renewcommand{\@evenfoot}{\@oddfoot}}
\makeatother     

\title{THE REGULARITY INDEX, HILBERT FUNCTION OF FAT POINTS AND THE EMBEDDING}         
\author{Phan Van Thien and Phan Quang Nhu Anh}
\date{}          
\maketitle 
\pagestyle{plain}

\begin{abstract}\noindent Let $m \ge n$, $\phi_{n,m}: \mathbb P^n \to \mathbb P^m$,   $\phi_{n,m}(a_1, \ldots, a_n)=(a_1, \ldots, a_n, 0, \ldots, 0)$, be the embedding, $Z=m_1P_1+\cdots+m_sP_s$ be fat points in $\mathbb P^n$ and $\phi_{n,m}(Z)=m_1\phi_{n,m}(P_1)+\cdots+m_s\phi_{n,m}(P_s)$ be fat points in $\mathbb P^m$. We show the relation between the regularity index, Hilbert function of the fat points $Z$ and the regularity index, Hilbert function of the fat points $\phi_{n,m}(Z)$. 
\end{abstract}

\noindent {\it Keywords:} Regularity index; fat points.

\par \noindent {Mathematics Subject Classification 2020:} 14C20, 13D40.

\par \section{Introduction} Denote by $\mathbb P^n$ the $n$-dimensional  projective space over an algebraically closed field $K$, and by $R_n=K[X_0, \ldots, X_n]$ the polynomial ring in variables $X_0, \ldots, X_n$ over $K$. Let $P_1, \ldots, P_s$ be distinct
points in $\mathbb P^n$, and let $m_1, \ldots, m_s$ be positive integers. Denote by $\wp_i$ the  defining homogeneous prime ideal of $P_i$ in $R$, $i=1, \ldots, s$. We call the zero-scheme defined by the ideal $\wp^{m_1}_1 \cap \cdots \cap
\wp^{m_s}_s$ to be a set of fat points in $\mathbb P^n$ (we also call it to be fat points in $\mathbb P^n$), and denote it by  $$Z=m_1P_1+\cdots+m_sP_s.$$ 
Put $I_Z=\wp^{m_1}_1 \cap \cdots \cap
\wp^{m_s}_s$,  is the defining ideal of $Z$. The ring $R_n/I_Z$ is the  homogeneous coordinate ring of $Z$. It is graded, $R_n/I_Z=\underset{t
\ge 0}{\oplus} (R_n/I_Z)_t$, and it has the multiplicity to be 
$e(R_n/I_Z)
=\underset{i=1}{\overset{s}{\sum}} \binom{m_i+n-1}{n}$. For each $t$, the $t$-th graded part $(R_n/I_Z)_t$ is a finite dimensional $K$-vector space. The function $$H_{R_n/I_Z}(t)=\dim_K (R_n/I_Z)_t$$
is called the Hilbert function of $R_n/I_Z$. This function strictly increases until it reaches the multiplicity $e(R_n/I_Z)$, at which it stabilizes. The number 
$$\reg(R_n/I_Z)=\min\{t\in \mathbb N \ | \ H_{R_n/I_Z}(t)=e(R_n/I_Z)\}$$
is called the regularity index of the coordinate ring $R_n/I_Z$ (we call also it to be the regularity index of $Z$, and denote it by $\reg(Z)$).

Let $m$ be an integer, $m >n$,  and let 
\begin{align*}
\phi_{n,m}: \mathbb P^n &\to \mathbb P^m\\
P=(a_1, \ldots, a_n) &\mapsto \phi_{n,m}(P)=(a_1, \ldots, a_n, 0, \ldots, 0)
\end{align*}
be the embedding. Put $R_m=K[X_0, \ldots, X_n, \ldots, X_m]$ and denote by $\q_i$ the  defining homogeneous prime ideal of $\phi_{n,m}(P_i)$, $i=1, \ldots, s$. Put $$\phi_{n,m}(Z)=m_1\phi_{n,m}(P_1)+\cdots+m_s\phi_{n,m}(P_s)$$ to be the fat points defined by $\q_1^{m_1}\cap \cdots \cap \q_s^{m_s}$. 
The coordinate ring of $\phi_{n,m}(Z)$ is $R_m/I_{\phi_{n,m}(Z)}$, it is graded, $R_m/I_{\phi_{n,m}(Z)} =\underset{t\ge 0}{\oplus} (R_m/I_{\phi_{n,m}(Z)})_t$. The Hilbert function  of $R_m/I_{\phi_{n,m}(Z)}$ (or of $\phi_{n,m}(Z)$) is 
$$H_{R_m/I_{\phi_{n,m}(Z)}}(t)=\dim_K (K[X_0, \ldots, X_m]/I_{\phi_{n,m}(Z)})_t.$$
This function strictly increases until it reaches the multiplicity $e(R_m/I_{\phi_{n,m}(Z)})=\underset{i=1}{\overset{s}{\sum}} \binom{m_i+m-1}{m}$, at which it stabilizes.  The regularity index of $R_m/I_{\phi_{n,m}(Z)}$ (or of fat points $\phi_{n,m}(Z)$) is the number
$$\reg(R_m/I_{\phi_{n,m}(Z)})=\min\{t\in \mathbb N \ | \ H_{R_m/I_{\phi_{n,m}(Z)}}(t)=e(R_m/I_{\phi_{n,m}(Z)})\}.$$
In this paper, we first show that the regularity index of fat points is invariant under the embedding, i.e.
$$\reg(R_n/I_Z) = \reg(R_m/I_{\phi_{n,m}(Z)})  \quad \text{ (Theorem \ref{thm37})}$$
 by using the algebra method of Catalisano, Trung and Valla in \cite{CTV}. This theorem shows  that some results about the regularity index of the set of non-degenerate fat points (see \cite{BDP}, \cite{BFL}, ...) still hold for a larger class of fat points. Next we show that the Hilbert function of fat points is not invariant under the embedding,  they have the relation as following:  $$H_{R_m/I_{\phi_{n,m}(Z)}}(t)  \ge  H_{R_n/I_Z}(t)$$
with $t\ge \reg(R_n/I_Z)$ (Proposition \ref{prop41}) and 
\begin{align*}H_{R_{m}/I_{\phi_{n,m}(Z)}}(t) &= H_{R_n/I_Z}(t) + \binom{t+m}{m} - \binom{t+n}{n} \\
	&- \underset{i=0}{\overset{t-1}{\sum}}\binom{m-n-1+t-i}{t-i}  \left(   \binom{i+n}{n} - H_{R_n/(\wp_1^{m_1-t+i} \cap \cdots \cap \wp_s^{m_s-t+i})}(i)\right) \end{align*}
with $0\le  t<\reg(R_n/I_Z)$ (Theorem \ref{thm45}). Then we show that 
$$H_{R_{m}/I_{\phi_{n,m}(Z)}}(t) \ge H_{R_n/I_Z}(t)$$
with $t\ge 0$, and if there exists $m_j\ge 2$, then $$H_{R_{m}/I_{\phi_{n,m}(Z)}}(t) > H_{R_n/I_Z}(t)\quad (\text{Corollary }\ref{cor46}).$$
The above results are the answer to Question 1 in \cite{Thien}. In addition, using the above results and the known results about the Hilbert function of fat points (see  \cite{A1}-\cite{A4},\cite{CEG}, \cite{H}, ...) we can get results about the Hilbert function of another  fat points.

\section{Preliminaries} 

Denote by $\wp_i \subset K[X_0, \ldots, X_n]$ the prime ideal defining by $P_i \in \mathbb P^n$. The following lemmas have been proved in  \cite{CTV}.

\begin{lemma}\label{lem21} \cite[Lemma 1]{CTV} Let $P_1,\ldots , P_s$ be
distinct points in $\mathbb P^n$ and $m_1,\ldots , m_s$ are positive integers. Put $J = \wp^{m_1}_1\cap
\cdots\cap  \wp_{s-1}^{m_{s-1}}$ and $I = J \cap \wp_s^{m_s}$. Then $$\reg(R_n/I) =
\max\left\{m_s-1, \reg(R_n/J), \reg(R_n/(J+\wp_s^{m_s-1})) \right\}.$$ \end{lemma}

To estimate  $\reg(R_n/(J+\wp^{m_s}))$ we shall  use the following lemma. \par
\begin{lemma}\label{lem22} \cite[Lemma 3]{CTV}
Let $P_1,\ldots , P_s$ be
distinct points in $\mathbb P^n$ and $m_1,\ldots , m_s$ are positive integers. If $J = \wp^{m_1}_1\cap
\cdots\cap  \wp_{s-1}^{m_{s-1}}$ and $\wp_s = (X_1,\ldots , X_n)$, then
$$\reg(R_n/(J+\wp_s^{m_s})) \le b$$ if and only if $X_0^{b-i}M \in J+\wp_s^{i+1}$ for
every monomial $M$ of degree $i$ in $X_1,\ldots , X_n$; $i = 0,\ldots , m_s-1$.
\end{lemma}  \par

\begin{lemma}\label{lem23} \cite[Corollary 2]{CTV}
 Let $s\ge 2$; $P_1, \ldots, P_s$ be distinct points in $\mathbb P^n$ and $m_1 \ge \cdots \ge m_s$ be positive integers. If $I=\wp_1^{m_1}\cap \cdots \cap  \wp_s^{m_s}$, then 
 $$\reg(R_n/I)\ge m_1+m_2-1.$$
\end{lemma}  \par

 \par

\section{Invariant of regularity index under embedding}

From now on we use: for positive integers $m\ge n$, denote by $\phi_{n,m}: \mathbb P^n \to \mathbb P^m$  the embedding. If $P\in \mathbb P^n$, then denote by $\wp$ the defining prime ideal of $P$ and denote by $\q$ the defining prime ideal of $\phi(P) \in \mathbb P^m$.

\begin{lemma}\label{lem31} Let $\phi_{n,n+1}: \mathbb P^n \to \mathbb P^{n+1}$ be the embedding, $P\in \mathbb P^n$ and $r$ be a positive integer.  We have

i) If $f\in \q^r$, then $f(X_0, \ldots, X_n, 0)\in \wp^r$.

ii) $\wp^r = \q^r \cap R_n.$
\end{lemma} 
\begin{proof}   Let $P=(a_{0}, \ldots, a_{n})\in \mathbb P^n$. Then there exist linear forms $l_{0}, \ldots, l_{n} \in R_n=K[X_0, \ldots, X_{n}]$ such that $\wp=(l_{0}, \ldots, l_{n})$.
Since $\phi_{n,n+1}(P)=(a_{0}, \ldots, a_{n}, 0)\in \mathbb P^{n+1}$, we have $\q=(l_{0}, \ldots, l_{n}, X_{n+1})$. 

i) We argue by induction on $r$. If $f\in \q=(l_{0}, \ldots, l_{n}, X_{n+1})$, then $f$ can be written $$f=\sum_{i=0}^nl_if_i+X_{n+1}f_{n+1}$$
with $f_i\in R_{n+1}$, $i=0, \ldots, n+1$. We have 
\begin{align*}
f(X_0, \ldots, X_n, 0)&=\sum_{i=0}^nl_i(X_0, \ldots, X_n)f_i(x_0, \ldots, X_n, 0)+0f_{n+1}(X_0, \ldots, X_n, 0)\\
&\in (l_0, \ldots, l_n)=\wp.
\end{align*}
So, i) is true for $r=1$. Now we assume that i) is true for $r-1$. If $f\in \q^r=\q\q^{r-1}$, then $f$ can be written $$f=\sum_{i=1}^d g_ih_i$$
with some $d\in \mathbb N^*$, $g_i\in \q$ and $h_i\in \q^{r-1}$, $i=1, \ldots, d$. We have $g_i(X_0, \ldots, X_n, 0) \in \wp$ by the above result, and $h_i(X_0, \ldots, X_n, 0)\in \wp^{r-1}$ by inductive assumption. Therefore,
$$f(X_0, \ldots, X_n, 0)=\sum_{i=1}^d g_i(X_0, \ldots, X_n, 0)h_i(X_0, \ldots, X_n, 0) \in \wp \wp^{r-1}=\wp^r.$$

ii) We have $\wp=(l_{0}, \ldots, l_{n})$ and $\q=(l_{0}, \ldots, l_{n}, X_{n+1})$ with linear forms $l_{0}, \ldots, l_{n} \in R_n=K[X_0, \ldots, X_{n}]$. It is clear that $$\wp^{r} \subset \q^{r} \cap R_{n}.$$
For any $g\in \q^r \cap R_n$, since $g\in \q^r=\q\q^{r-1}$, there exists a positive integer $d$ such that $$g=\sum_{i=1}^d h_ik_i$$
with $h_i\in \q$ and $k_i\in \q^{r-1}$, $i=1, \ldots, d$. By i) we have  $h_i(X_0, \ldots, X_n, 0)\in \wp$ and $k_i(X_0, \ldots, X_n, 0)\in \wp^{r-1}$, $i=1, \ldots, d$. Since $g\in  \q^r \cap R_n$, we have $g=g(X_0, \ldots, X_n, 0)$. Therefore
\begin{align*}
g&=g(X_0, \ldots, X_n, 0)\\
&=\sum_{i=1}^d h_i(X_0, \ldots, X_n, 0)k_i(X_0, \ldots, X_n, 0) \in \wp \wp^{r-1}=\wp^r.
\end{align*}
So,  $\q^{r} \cap R_{n} \subset \wp^{r}$.
\end{proof}

\begin{corollary}\label{cor32} Let $\phi_{n,n+1}: \mathbb P^n \to \mathbb P^{n+1}$ be the embedding and $Z=m_1P_1+\cdots+m_sP_s$ be  fat points in $\mathbb P^n$. Then 

i) If $f\in I_{\phi_{n,n+1}(Z)}$, then $f(X_0, \ldots, X_n, 0) \in I_Z$.

ii) $I_Z = I_{\phi_{n,n+1}(Z)}\cap R_n$.
\end{corollary}
\begin{proof} Recall that $I_Z=\wp_1^{m_1}\cap \cdots \cap \wp_s^{m_s}$ is the defining ideal of $Z$, and $I_{\phi_{n,n+1}(Z)}=\q_1^{m_1}\cap \cdots \cap \q_s^{m_s}$ is the defining ideal of $\phi_{n,n+1}(Z)$.  

i)  If $f\in I_{\phi_{n,n+1}(Z)}$, then $f\in \q_i^{m_i}$, $i=1, \ldots, s$. By Lemma \ref{lem31}, i) we have $f(X_0, \ldots, X_n, 0) \in \wp_i^{m_i}$, $i=1, \ldots, s$. It follows that $$f(X_0, \ldots, X_n, 0) \in \wp_1^{m_1}\cap \cdots \cap \wp_s^{m_s}=I_Z.$$

ii) For $i=1, \ldots, s$, by the Lemma \ref{lem31}, ii) we get $$\wp_i^{m_i}=\q_i^{m_i}\cap R_n.$$ 
It follows that
$$I_Z=\wp_1^{m_1}\cap \cdots \cap \wp_s^{m_s}=\q_1^{m_1}\cap \cdots \cap \q_s^{m_s} \cap R_n=I_{\phi_{n,n+1}(Z)}\cap R_n.$$
\end{proof}

\begin{corollary}\label{cor33} Let $m>n$ be positive integers, $\phi_{n,m}: \mathbb P^n \to \mathbb P^{m}$ be the embedding, $P\in \mathbb P^n$ and $r$ be a positive integer.  We have

i) If $f\in \q^r$, then $f(X_0, \ldots, X_n, 0, \ldots, 0)\in \wp^r$.

ii) $\wp^r = \q^r \cap R_m.$
\end{corollary}

\begin{proof} Consider $m-n$ the embedding maps $$\phi_{i,i+1}: \mathbb P^i \to \mathbb P^{i+1},$$
$i=n, \ldots, m-1$. Use the Lemma \ref{lem31} for each the embedding $\phi_{i,i+1}: \mathbb P^i \to \mathbb P^{i+1}$, $i=n, \ldots, m-1$, and 
$$\phi_{n,m} = \phi_{m-1,m} \circ \cdots \circ \phi_{n,n+1}$$
we get i) and ii).
\end{proof}

\begin{corollary}\label{cor34} Let   $m>n$ be positive integers, $\phi_{n,m}: \mathbb P^n \to \mathbb P^{m}$ be the embedding and $Z=m_1P_1+\cdots+m_sP_s$ be fat points in $\mathbb P^n$. Then 

i) If $f\in I_{\phi_{n,m}(Z)}$, then $f(X_0, \ldots, X_n, 0, \ldots, 0) \in I_Z$.

ii) $I_Z = I_{\phi_{n,m}(Z)}\cap R_n$.
\end{corollary}
\begin{proof} Consider $m-n$ the embedding maps $$\phi_{i,i+1}: \mathbb P^i \to \mathbb P^{i+1},$$
$i=n, \ldots, m-1$. Use the Corollary \ref{cor32} for each the embedding $\phi_{i,i+1}: \mathbb P^i \to \mathbb P^{i+1}$, $i=n, \ldots, m-1$, and 
$$\phi_{n,m} = \phi_{m-1,m} \circ \cdots \circ \phi_{n,n+1}$$
we get i) and ii).
\end{proof}

\begin{proposition}\label{prop35}    Let $m>n$ be positive integers and $Z=m_1P_1+\cdots+m_sP_s$ be fat points in $\mathbb P^n$. Let $\phi_{n,n+1}: \mathbb P^n \to \mathbb P^{n+1}$ be the embedding.  Put $U=m_1P_1+\cdots+m_{s-1}P_{s-1}$. Then 
$$\reg(R_n/(I_U+\wp_s^{m_s})) = \reg(R_{n+1}/(I_{\phi_{n,n+1}(U)}+\q_s^{m_s})).$$
\end{proposition}
\begin{proof} We have $I_U=\wp_1^{m_1} \cap \cdots \cap \wp_{s-1}^{m_{s-1}}$ and $I_{\phi_{n,n+1}(U)}=\q_{1}^{m_1} \cap \cdots \cap \q_{s-1}^{m_{s-1}}$. Put $P_s=(1, 0, \ldots, 0)$, then $\wp_s=(X_1, \ldots, X_n) \subset R_n$, $\phi_{n,n+1}(P_{s})=(1, 0, \ldots, 0)\in \mathbb P^{n+1}$ and $\q_{s}=(X_1, \ldots, X_{n+1})\subset R_{n+1}$. 

To prove the proposition we prove $\reg(R_n/(I_U+\wp_s^{m_s})) \ge \reg(R_{n+1}/(I_{\phi_{n,n+1}(U)}+\q_s^{m_s}))$ and $\reg(R_n/(I_U+\wp_s^{m_s})) \le \reg(R_{n+1}/(I_{\phi_{n,n+1}(U)}+\q_s^{m_s})).$

\medskip
\noindent {\it 1. Prove $\reg(R_n/(I_U+\wp_s^{m_s})) \ge \reg(R_{n+1}/(I_{\phi_{n,n+1}(U)}+\q_s^{m_s}))$}:
We prove that if there is a non-negative integer $t$ such that $$\reg(R_{n+1}/(I_{\phi_{n,n+1}(U)}+\q_s^{m_s})) > t,$$
 then
$$\reg(R_n/(I_U+\wp_s^{m_s})) > t.$$
 If 
$\reg(R_{n+1}/(I_{\phi_{n,n+1}(U)}+\q_s^{m_s})) > t,$
then by Lemma \ref{lem22} there exists a monomial $M=X_1^{c_1} \cdots X_{n+1}^{c_{n+1}}$, $c_1+\cdots+c_{n+1}=i$, $i\in \{0, \ldots, m_s-1\}$, such that $$X_0^{t-i}M \notin I_{\phi_{n,n+1}(U)}+\q_{s}^{i+1}.\qquad (1)$$
By Lemma \ref{lem31}, ii) we get $\wp_{s}^{i+1} = \q_s^{i+1}\cap R_n$. By Corollary \ref{cor32}, ii) we get
$I_U= I_{\phi_{n,n+1}(U)} \cap R_{n}$. Therefore, $$I_U+\wp_{s}^{i+1} \subset I_{\phi_{n,n+1}(U)}+\q_{s}^{i+1}.\qquad (2)$$ 
Consider monomial $N=X_1^{c_1} \cdots X_{n}^{c_{n}}$. If $X_0^{t-i}N \in I_U+\wp_{s}^{i+1}$, then  $$X_{n+1}^{c_{n+1}}X_0^{t-i}N\in I_U+\wp_{s}^{i+1}.\qquad \qquad (3)$$ From $(2)$ and $(3)$ we get 
$$X_0^{t-i}M=X_{n+1}^{c_{n+1}} X_0^{t-i}N\in I_{\phi_{n,n+1}(U)}+\q_s^{i+1},$$
which is contradiction with $(1)$. So, $$X_0^{t-i}N \notin I_U+\wp_{s}^{i+1}.$$
Then by Lemma \ref{lem22} we get
$$\reg(R/(J_\alpha+\wp_{s\alpha}^{m_s})) > t.$$
\medskip
\noindent {\it 2. Prove $\reg(R_n/(I_U+\wp_s^{m_s})) \le \reg(R_{n+1}/(I_{\phi_{n,n+1}(U)}+\q_s^{m_s}))$}: We prove that if there is a non-negative integer $t$ such that $$\reg(R_n/(I_U+\wp_s^{m_s})) > t,$$
 then
$$\reg(R_{n+1}/(I_{\phi_{n,n+1}(U)}+\q_s^{m_s})) > t.$$
If $\reg(R_n/(I_U+\wp_s^{m_s})) > t$, then by Lemma \ref{lem22} there exists a monomial $M=X_1^{c_1} \cdots X_{n}^{c_{n}}$, $c_1+\cdots+c_{n}=i$, $i\in \{0, \ldots, m_s-1\}$, such that $$X_0^{t-i}M \notin I_U+\wp_s^{i+1}.\qquad (4)$$
If $X_0^{t-i}M \in I_{\phi_{n,n+1}(U)}+\q_s^{i+1}$, then there are $f\in I_{\phi_{n,n+1}(U)}$ and $g\in 
\q_s^{i+1}$ such that $$X_0^{t-i}M = f + g.$$
 Since $f\in I_{\phi_{n,n+1}(U)}$,  we have $f(X_0, \ldots, X_n, 0)\in I_U$ by using Corollary \ref{cor32}, i). Since $g\in \q_s^{i+1}$,  we have $g(X_0, \ldots, X_n, 0)\in \wp^{i+1}$ by using Lemma \ref{lem31}, i). Then since $X_0^{t-i}M\in R_n$, we have 
$$X_0^{t-i}M = f(X_0, \ldots, X_n, 0)+ g(X_0, \ldots, X_n, 0)\in I_U + \wp^{i+1},$$
which is contradiction with $(4)$. So, $$X_0^{t-i}M \notin  I_{\phi_{n,n+1}(U)}+\q_s^{i+1}.$$
By Lemma \ref{lem22} we get
$$\reg(R_{n+1}/(I_{\phi_{n,n+1}(U)}+\q_s^{m_s})) > t.$$

\end{proof}

\begin{proposition}\label{prop36} Let $m>n$ be positive integers and $Z=m_1P_1+\cdots+m_sP_s$ be fat points in $\mathbb P^n$. Let $\phi_{n,n+1}: \mathbb P^n \to \mathbb P^{n+1}$ be the embedding. Then
$$\reg(R_n/I_Z) = \reg(R_{n+1}/I_{\phi_{n,n+1}(Z)}).$$
\end{proposition}
\begin{proof} We prove the proposition by induction on $s$. Without loss of generality we may assume that $m_1 \ge \ldots \ge m_s$. 

If $s=1$, then $Z=m_1P_1$, $I_Z=\wp_1^{m_1}\subset R_n$ and $\phi_{n,n+1}(Z)=m_1\phi_{n,n+1}(P_1)$, $I_{\phi_{n,n+1}(Z)})=\q_1^{m_1}\subset R_{n+1}$. It is well known that $\reg(R_n/\wp_1^{m_1})=m_1-1$ and $\reg (R_{n+1}/\q_1^{m_1})=m_1-1$. Thus, the proposition is true for $s=1$. 

Let $s\ge 2$, assume that the proposition is true for $s-1$. Put $$U=m_1P_1+\cdots+m_{s-1}P_{s-1}.$$
Then $I_U=\wp_1^{m_1}\cap \cdots \cap \wp_{s-1}^{m_{s-1}}$, $I_Z=I_U \cap \wp_{s}^{m_s}$,  $I_{\phi_{n,n+1}(U)}=\q_1^{m_1}\cap \cdots \cap \q_{s-1}^{m_{s-1}}$ and $I_{\phi_{n,n+1}(Z)}=I_{\phi_{n,n+1}(U)}\cap \q_{s}^{m_s}$. By Lemma \ref{lem21} we get 
 $$\reg(R_n/I_Z) =
\max\left\{m_s-1, \reg(R_n/I_U), \reg(R_n/(I_U+\wp_s^{m_s})) \right\}$$
and
$$\reg(R_{n+1}/I_{\phi_{n,n+1}(Z)}) =
\max\left\{m_s-1, \reg(R_{n+1}/I_{\phi_{n,n+1}(U)}), \reg(R_{n+1}/(I_{\phi_{n,n+1}(U)}+\q_{s}^{m_s})) \right\}.$$

By Lemma \ref{lem23} we have $\reg(R_n/I_Z) \ge m_1+m_2-1$ and $\reg(R_{n+1}/I_{\phi_{n,n+1}(Z)}) \ge m_1+m_2-1$. So, 
$$\reg(R_n/I_Z) =
\max\left\{\reg(R_n/I_U), \reg(R_n/(I_U+\wp_s^{m_s})) \right\}$$
and
$$\reg(R_{n+1}/I_{\phi_{n,n+1}(Z)}) =
\max\left\{\reg(R_{n+1}/I_{\phi_{n,n+1}(U)}), \reg(R_{n+1}/(I_{\phi_{n,n+1}(U)}+\q_{s}^{m_s})) \right\}.$$
By inductive assumption we have
$$\reg(R_n/I_U) = \reg(R_{n+1}/I_{\phi_{n,n+1}(U)}).$$
By Proposition \ref{prop35} we have
$$\reg(R_n/(I_U+\wp_s^{m_s}))= \reg(R_{n+1}/(I_{\phi_{n,n+1}(U)}+\q_{s}^{m_s})).$$ Therefore, 
$$\reg(R_n/I_Z) = \reg(R_{n+1}/I_{\phi_{n,n+1}(Z)}).$$
\end{proof}

\medskip

From the result of Benedetti et al. in \cite[Lemma 4.4]{BFL} we can deduce 
$$\reg(R_n/I_Z) \ge \reg(R_{n+1}/I_{\phi_{n,n+1}(Z)}),$$
but it is difficult by using the method in \cite{BFL}  to show the equality.

\begin{theorem}\label{thm37} Let $m>n$ be positive integers, $Z=m_1P_1+\cdots+m_sP_s$ be fat points in $\mathbb P^n$ and $\phi_{n,m}: \mathbb P^n \to \mathbb P^{m}$ be the embedding. Then
$$\reg(R_n/I_Z) = \reg(R_m/I_{\phi_{n,m}(Z)}).$$
\end{theorem}
\begin{proof} Consider $m-n$ the embedding maps $$\phi_{i,i+1}: \mathbb P^i \to \mathbb P^{i+1},$$
$i=n, \ldots, m-1$. By Proposition \ref{prop36} we have 
\begin{align*}
\reg(R_n/I_Z) &= \reg(R_{n+1}/I_{\phi_{n,n+1}(Z)}),\\
\reg(R_{n+1}/I_{\phi_{n,n+1}(Z)})&= \reg(R_{n+2}/I_{\phi_{n+1,n+2} \circ \phi_{n,n+1}(Z)}),\\
\cdots & \cdots \cdots \\
\reg(R_{i+1}/I_{\phi_{i,i+1}\circ \cdots \circ \phi_{n,n+1}(Z)})&= \reg(R_{i+2}/I_{\phi_{i+1,i+2} \circ \cdots \circ \phi_{n,n+1}(Z)}),\\
\cdots & \cdots \cdots \\
\reg(R_{m-1}/I_{\phi_{m-2,m-1}\circ \cdots \circ \phi_{n,n+1}(Z)})&= \reg(R_{m}/I_{\phi_{m-1,m} \circ \cdots \circ \phi_{n,n+1}(Z)}).\\
\end{align*}
These follow that $$\reg(R_n/I_Z)=\reg(R_{m}/I_{\phi_{m-1,m} \circ \cdots \circ \phi_{n,n+1}(Z)}).$$ 
But $\phi_{m-1,m} \circ \cdots \circ \phi_{n,n+1}=\phi_{n,m}$. Thus, we get
$$\reg(R_n/I_Z)=\reg(R_m/I_{\phi_{n,m}(Z)}).$$
\end{proof}

\medskip
The configuration of fat points $\phi_{n,m}(Z)$ in $\mathbb P^m$ may be different to the configuration of fat points $Z$ in $\mathbb P^n$. Use the above theorem we can know the regularity index of the coordinate ring of $\phi_{n,m}(Z)$ when we know the regularity index of the coordinate ring of $Z$.

\medskip
\begin{example} \textnormal{A rational normal curve in $\mathbb P^n$ to be a curve of degree $n$ that may be given parametrically as the image of the map}
	\begin{align*} \mathbb P^1 &\to \mathbb P^n.\\
		(s, t) &\mapsto (s^n, s^{n-1}t, \ldots, t^n)
	\end{align*}

	\textnormal{Let $Z=m_1P_1+\cdots+m_sP_s$ be fat points in $\mathbb P^n$ with $m_1 \ge \cdots \ge m_s$ and $P_1, \ldots, P_s$ are on a rational normal curve of $\mathbb P^n$. Catalisano et al. showed that \cite[Proposition 7]{CTV}}
	$$\reg(R_n/I_Z) = \max\left\{m_1+m_2-1, \left[(\sum_{i=1}^s
	m_i+n-2)/n\right]\right\},$$ 
	\textnormal{where $\left[(\sum_{i=1}^s m_i+n-2)/n\right]$ is the integer part of the rational number $\sum_{i=1}^s m_i+n-2)/n$.} 
		
\textnormal{Let  $m>n$ be positive integers and $\phi_{n,m}: \mathbb P^n \to \mathbb P^m$ be the embedding. If $s > n+1$, then  $\phi_{n,m}(P_1), \ldots, \phi_{n,m}(P_s)$ are not on a rational normal curve in $\mathbb P^m$. So, we can not use  \cite[Proposition 7]{CTV} to know $\reg(R_m/I_{\phi_{n,m}(Z)})$. But now by using Theorem \ref{thm37} we get}
	$$\reg(R_m/I_{\phi_{n,m}(Z)}) = \reg(R_n/Z)= \max\left\{m_1+m_2-1, \left[(\sum_{i=1}^s
	m_i+n-2)/n\right]\right\}.$$
\textnormal{(also see \cite[Corollary 5.5]{NT} or \cite[Proposition 10]{Th5} for the similar results).}
\end{example}

\medskip

\section{Is the Hilbert function invariant under embedding?}
 Let  $m > n$ be positive integers,  $Z=m_1P_1+\cdots+m_sP_s$ be fat points in $\mathbb P^n$
and $\phi_{n,m}: \mathbb P^n \to \mathbb P^m$ be the embedding. What  is the relation between $H_{R_m/I_{\phi_{n,m}(Z)}}(t)$ and $H_{R_n/I_Z}(t)$? 
So far, we know only that $$H_{R_n/I_Z}(t)=\binom{t+n}{n}-\dim_K (I_Z)_t$$
and it strictly increases until it reaches the multiplicity $e(R_n/I_Z)=\underset{i=1}{\overset{s}{\sum}} \binom{m_i+n-1}{n}$, at which it stabilizes. Similarly, $$H_{R_m/I_{\phi_{n,m}(Z)}}(t)=\binom{t+m}{m}-\dim_K (I_{\phi_{n,m}(Z)})_t$$
and it strictly increases until it reaches the multiplicity $e(R_m/I_{\phi_{n,m}(Z)})=\underset{i=1}{\overset{s}{\sum}} \binom{m_i+m-1}{m}$, at which it stabilizes. Since $m>n$, we have 
$$\underset{i=1}{\overset{s}{\sum}} \binom{m_i+m-1}{m} \ge \underset{i=1}{\overset{s}{\sum}} \binom{m_i+n-1}{n},$$
the equality occurs if and only if $m_1=\cdots=m_s=1$.
By Theorem \ref{thm37} we have
$\reg(R_n/I_Z) =\reg(R_m/I_{\phi_{n,m}(Z)})$. So if $t\ge \reg(R_n/I_Z)$, then
$$H_{R_m/I_{\phi_{n,m}(Z)}}(t) = \underset{i=1}{\overset{s}{\sum}} \binom{m_i+m-1}{m}$$
and $$H_{R_n/I_Z}(t) = \underset{i=1}{\overset{s}{\sum}} \binom{m_i+n-1}{n}.$$
Thus, we get the following result.
\begin{proposition}\label{prop41} Let  $m > n$ be positive integers,  $Z=m_1P_1+\cdots+m_sP_s$ be fat points in $\mathbb P^n$
and $\phi_{n,m}: \mathbb P^n \to \mathbb P^m$ be the embedding. If $t\ge \reg(R_n/I_Z)$, then
$$H_{R_m/I_{\phi_{n,m}(Z)}}(t)= \underset{i=1}{\overset{s}{\sum}} \binom{m_i+m-1}{m} \ge  H_{R_n/I_Z}(t)= \underset{i=1}{\overset{s}{\sum}} \binom{m_i+n-1}{n}.$$
The equality occurs if and only if  $m_1=\cdots=m_s=1$.
\end{proposition}

\medskip
What is about the relation between $H_{R_m/I_{\phi_{n,m}(Z)}}(t)$ and $H_{R_n/I_Z}(t)$ when $t<\reg(R_n/I_Z)$? 

From now on let  $\wp^r =R_n$ if $r$ is a non-positive integer and $\wp$ is the defining ideal of a point $P\in \mathbb P^n$.

\begin{lemma}\label{lem42} Let $Z=m_1P_1+\cdots+m_sP_s$ be a set of  fat points in $\mathbb P^n$ and $\phi_{n,m}: \mathbb P^n \to \mathbb P^{m}$ be the embedding. If $0\le t<\reg(Z)$ and $f \in [I_{\phi_{n,m}(Z)}]_t$, then $f$ can be written in the form
$$f=_tf + _{t-1}fh_1 + \cdots + _{1}fh_{t-1} + h_{t},$$
where $_if \in [\wp_1^{m_1+i-t}\cap \cdots \cap\wp_s^{m_s+i-t}]_i $,  $h_i\in [K[X_{n+1}, \ldots, X_m]]_i$, $i= 1, \ldots, t$.
\end{lemma}
\begin{proof} If $f\in [I_{\phi_{n,n+1}(Z)}]_t$, then $f\in (K[X_0, \ldots, X_n])[X_{n+1}, \ldots, X_m]$. So $f$ can be written in the form
$$f= _tf + _{t-1}fh_1 + \cdots + _ifh_{t-i} + _{i-1}fh_{t-i-1} +\cdots + _{1}fh_{t-1} + h_{t},  \qquad (5)$$
where $_if \in [R_n]_i=[K[X_0, \ldots, X_n]]_i$, $h_i\in [K[X_{n+1}, \ldots, X_m]]_i$, $i= 1, \ldots, t$.  We prove that $_if \in  [\wp_1^{m_1-t+i} \cap \cdots \cap \wp_s^{m_s-t+i}]_i$, $i=1, \ldots, t$.

For $i=1, \ldots, t$, since $h_i\in [K[X_{n+1}, \ldots, X_m]]_i$, we get $h_i( 0, \ldots, 0)=0$. Thus, 
\begin{align*}
	 f(X_0, \ldots, X_n, 0, \ldots, 0) &= _tf(X_0, \ldots, X_n) +  _{t-1}f(X_0, \ldots, X_n)h_1(0, \ldots, 0) + \cdots   + h_{t}(0, \ldots, 0)\\
	 &=_tf (X_0, \ldots, X_n)+ _{t-1}f(X_0, \ldots, X_n)0 + \cdots  + 0\\
	 &= _tf(X_0, \ldots, X_n) .
\end{align*}

Since $f\in I_{\phi_{n,m}(Z)}$, by Corollary \ref{cor34} we get $f(X_0, \ldots, X_n, 0, \ldots, 0) \in I_Z$. Thus,
$$ _tf = _tf(X_0, \ldots, X_n)  = f(X_0, \ldots, X_n, 0, \ldots, 0) \in I_Z.$$
Note that $\deg (_tf)=t$, so we get $_tf\in [I_Z]_t$.

For $i=t-1, \ldots, 1$, if $_ifh_{t-i}\ne 0$ then $h_{t-i}$ has a non-zero term, say  $aX_{n+1}^{c_1} \cdots X_m^{c_m}$. Take the derivative to order $t-i$ with respect to the variables $X_{n+1}, \ldots, X_m$ of $_tf$ and $_jfh_{t-j}$, $j=t-1, \ldots, i+1$, we get:
$$\frac{\partial^{t-i}(_tf)}{\partial^{{c_1}}X_{n+1} \cdots \partial^{{c_m}}X_{m}} =0, \frac{\partial^{t-i}(_jfh_{t-j})}{\partial^{{c_1}}X_{n+1} \cdots \partial^{{c_m}}X_{m}} =0.$$
Thus, take the derivative to order $t-i$ with respect to the variables $X_{n+1}, \ldots, X_m$ on both sides of (5) we get:
$$\frac{\partial^{t-i}f}{\partial^{{c_1}}X_{n+1} \cdots \partial^{{c_m}}X_{m}} = (_if)c_1! \cdots c_m!+ (_{i-1}f)  \left( \frac{\partial^{t-i}h_{t-i+1}}{   \partial^{{c_1}}X_{n+1} \cdots \partial^{{c_m}}X_{m}}\right) + \cdots + \frac{\partial^{t-i}h_{t}}{   \partial^{{c_1}}X_{n+1} \cdots \partial^{{c_m}}X_{m}}.$$
For $j=t-i+1, \ldots, t$; if $\frac{\partial^{t-i}h_{j}}{   \partial^{{c_1}}X_{n+1} \cdots \partial^{{c_m}}X_{m}} \ne 0$, then $\frac{\partial^{t-i}h_{j}}{   \partial^{{c_1}}X_{n+1} \cdots \partial^{{c_m}}X_{m}} \in [K[X_{n+1}, \ldots, X_m]]_{j-t+i}$ with $j-t+i \ge 1$. Therefore, 
$$\frac{\partial^{t-i} f}{   \partial^{{c_1}}X_{n+1} \cdots \partial^{{c_m}}X_{m} }(X_0, \ldots, X_n, 0, \ldots, 0) =  c_1! \cdots c_m!_{i}f(X_0, \ldots, X_n)= c_1! \cdots c_m!_{i}f. \qquad (6)$$
We now prove that $\frac{\partial^{t-i} f}{  \partial^{{c_1}}X_{n+1} \cdots \partial^{{c_m}}X_{m}}(X_0, \ldots, X_n, 0, \ldots, 0) \in 
\wp_1^{m_1-t+i} \cap \cdots \cap \wp_s^{m_s-t+i}.$
It is well known that the set of all homogeneous polynomials in $R_{m}$ that vanish at $\phi_{n,m}(P_i)$ to order $m_i$; for $i=1, \ldots, s$; is the ideal $I_{\phi_{n,m}(Z)}$. 
Since $f\in I_{\phi_{n,m}(Z)}$, we have all derivatives of $f$ of order $\le m_j-1$ vanish at $\phi_{n,m}(P_j)$, $j=1, \ldots, s$. This implies that all derivatives of $\frac{\partial^{t-i} f}{  \partial^{{c_1}}X_{n+1} \cdots \partial^{{c_m}}X_{m}}$ of order $\le m_j-1-t+i$ vanish at $\phi_{n,m}(P_j)$, $j=1, \ldots, s$. So
 $$\frac{\partial^{t-i} f}{  \partial^{{c_1}}X_{n+1} \cdots \partial^{{c_m}}X_{m}} \in \q_1^{m_1-t+i} \cap \cdots \cap \q_s^{m_s-t+i}.$$
By Corollary \ref{cor34} we get
$$\frac{\partial^{t-i} f}{  \partial^{{c_1}}X_{n+1} \cdots \partial^{{c_m}}X_{m}}(X_0, \ldots, X_n, 0, \ldots, 0) \in 
\wp_1^{m_1-t+i} \cap \cdots \cap \wp_s^{m_s-t+i}.\qquad (7)$$
From $(6)$ and (7) we get 
$$_{i}f\in \wp_1^{m_1-t+i} \cap \cdots \cap \wp_s^{m_s-t+i}.$$
Note that $\deg (_if)=i$, so $_if \in [\wp_1^{m_1-t+i} \cap \cdots \cap \wp_s^{m_s-t+i}]_i$.
\end{proof} 

\begin{lemma}\label{lem43}  Let $Z=m_1P_1+\cdots+m_sP_s$ be fat points in $\mathbb P^n$ and $\phi_{n,m}: \mathbb P^n \to \mathbb P^{m}$ be the embedding. If $0\le i \le t <\reg(Z)$, $g \in [\wp_1^{m_1+i-t}\cap \cdots \cap\wp_s^{m_s+i-t}]_i $ and $h_{t-i}\in [R[X_{n+1}, \ldots, X_m]]_{t-i}$, then 
	$$gh_{t-i} \in [I_{\phi_{n,m}(Z)}]_t.$$
\end{lemma}
\begin{proof}
	First, we prove that for $j=1, \ldots, s$, if $f\in [\wp_j^{m_j+i-t}]_i$ then $$fh_{t-i}\in [\q_j^{m_j}]_t.$$
	In fact, if $t-i \ge m_j$, then $h_{t-i}\in [R[X_{n+1}, \ldots, X_m]]_{t-i} \subset \q_j^{m_j}$. It implies $fh_{t-i}\in \q_j^{m_j}$. Note that $\deg(f)=i$ and $\deg(h_{t-i})=t-i$, therefore $fh_{t-i}\in [\q_j^{m_j}]_t$. If $t-i<m_j$, then $f\in [\wp_j^{m_j+i-t}]_i \subset \q_j^{m_j+i-t}$ and $h_{t-i}\in  [R[X_{n+1}, \ldots, X_m]]_{t-i} \subset \q_j^{t-i}$. So,  $fh_{t-i}\in [\q^
	{m_j}]_t$.
	
	Now if  $g \in [\wp_1^{m_1+i-t}\cap \cdots \cap\wp_s^{m_s+i-t}]_i $, then $g\in [\wp_j^{m_j+i-t}]_i$ for $j=1, \ldots, s$. By using the above result we get $gh_{t-i}\in [\q_j^{m_j}]_t$ for $j=1, \ldots, s$. Therefore,
	$$gh_{t-i}\in [\q_1^{m_1}]_t \cap \cdots \cap [\q_s^{m_s}]_t =
	[I_{\phi_{n,m}(Z)}]_t.$$
\end{proof}

\begin{proposition}\label{prop44} Let $Z=m_1P_1+\cdots+m_sP_s$ be   fat points in $\mathbb P^n$ and $\phi_{n,m}: \mathbb P^n \to \mathbb P^{m}$ be the embedding.  If $0\le t < \reg (R_n/I_Z)$, then 
	$$\dim_K[I_{\phi_{n,m}(Z)}]_t =  \dim_K[I_Z]_t + \underset{i=0}{\overset{t-1}{\sum}} \binom{m-n+t-i}{t-i}\dim_K[\wp_1^{m_1-t+i} \cap \cdots \cap \wp_s^{m_s-t+i}]_i.$$	
\end{proposition}
\begin{proof}  For $i=0, 1, \ldots, t$; let $r_i=\dim_K [\wp_1^{m_1-t+i} \cap \cdots \cap \wp_s^{m_s-t+i}]_i$ and $\{_ig_1, \ldots,\ _ig_{r_i}\}$ be a basis of $K$-vector space $[\wp_1^{m_1-t+i} \cap \cdots \cap \wp_s^{m_s-t+i}]_i$, put $d_{t-i}=\dim_K K[X_{n+1}, \ldots, X_m]_{t-i}$ and let $\{_{t-i}k_1, \ldots, _{t-i}k_{d_{t-i}}\}$ be a basis of $K$-vector space 	$[K[X_{n+1}, \ldots, X_m]]_{t-i}$.Then 
	$$\{ (_ig_j)(_{t-i}k_l)   | i=0, \ldots, t; j=1, \ldots, r_i; l=1, \ldots, d_{t-i}\} \text{ is linear independent in } [R_{m}]_t \qquad (8).$$
By Lemma \ref{lem42}, each polynomial $f \in [I_{\phi_{n,m}(Z)}]_t$ can be written in the form
$$f=_tf + _{t-1}fh_1 + \cdots + _{1}fh_{t-1} + h_{t}=\sum_{i=0}^t (_if)(h_{t-i}),$$
where $_if \in [\wp_1^{m_1+i-t}\cap \cdots \cap\wp_s^{m_s+i-t}]_i $,  $h_i\in [K[X_{n+1}, \ldots, X_m]]_i$; $i= 1, \ldots, t$; $_0f=h_0=1$.
 Since $\{_ig_1, \ldots,\ _ig_{r_i}\}$ is a basis of $K$-vector space $[\wp_1^{m_1-t+i} \cap \cdots \cap \wp_s^{m_s-t+i}]_i$, each $_if$ can be written in the form
	$$_if = \sum_{l=1}^{r_i}(c_{i,l})(_ig_l),$$
where $c_{i,l}\in K$, $l=1, \ldots, r_i$. Since $\{_{t-i}k_1, \ldots, _{t-i}k_{d_{t-i}}\}$ be a basis of $K$-vector space \\	$[K[X_{n+1}, \ldots, X_m]]_{t-i}$, each $h_i$ can be written in the form
$$h_i=\sum_{u=1}^{d_i}(b_{i,u})(_{i}k_u),$$
where $b_{i,u}\in K$, $u=1, \ldots, d_i$.
Therefore, $f$ can be written in the form
$$f=  \sum_{i=0}^t \left( \sum_{l=1}^{r_i}(c_{i,l})(_ig_l) \right) \left( \sum_{u=1}^{d_{t-i}}(b_{{t-i},u})(_{{t-i}}k_u)\right),$$
where all $c_{i,l}, b_{b-i,u}\in K$. By using Lemma \ref{lem43}, we have all $(_ig_l)(_{t-i}k_u)\in [I_{\phi_{n,m}(Z)}]_t$. So,  
$$\left\lbrace  (_ig_l)(_{t-i}k_u)   | i=0, \ldots, t; l=1, \ldots, r_i; u=1, \ldots, d_{t-i}\right\rbrace  \text{ is a generated set of } [I_{\phi_{n,n+1}(Z)}]_t \qquad (9).$$
From $(8)$ and $(9)$ we have $\left\lbrace  (_ig_j)(_{t-i}k_l)   | i=0, \ldots, t; j=1, \ldots, r_i; l=1, \ldots, d_{t-i}\right\rbrace $ is a basis of $K$-vector space $[I_{\phi_{n,m}(Z)}]_t$ and 
\begin{align*}\dim_K[I_{\phi_{n,m}(Z)}]_t &=  \underset{i=0}{\overset{t}{\sum}} r_id_{t-i} = r_t + \underset{i=0}{\overset{t-1}{\sum}}d_{t-i} r_i\\
&=\dim_K[I_Z]_t+\underset{i=0}{\overset{t-1}{\sum}} 	\binom{m-n-1+t-i}{t-i}\dim_K\left[ \wp_1^{m_1-t+i} \cap \cdots \cap \wp_s^{m_s-t+i}\right] _i.\end{align*}
\end{proof}

\begin{theorem}\label{thm45} Let $Z=m_1P_1+\cdots+m_sP_s$ be fat points in $\mathbb P^n$ and $\phi_{n,m}: \mathbb P^n \to \mathbb P^{m}$ be the embedding.  If $0\le t < \reg (R_n/I_Z)$, then 
\begin{align*}H_{R_{m}/I_{\phi_{n,m}(Z)}}(t) &= H_{R_n/I_Z}(t) + \binom{t+m}{m} - \binom{t+n}{n} \\
	&- \underset{i=0}{\overset{t-1}{\sum}}\binom{m-n-1+t-i}{t-i}  \left( \binom{i+n}{n} - H_{R_n/(\wp_1^{m_1-t+i} \cap \cdots \cap \wp_s^{m_s-t+i})}(i)\right) .\end{align*}
\end{theorem}

\begin{proof} We have $$H_{R_{m}/I_{\phi_{n,m}(Z)}}(t) = \binom{t+m}{m}- \dim_K[I_{\phi_{n,m}(Z)}]_t$$ and $$H_{R_n/(\wp_1^{m_1-t+i} \cap \cdots \cap \wp_s^{m_s-t+i})}(i)=\binom{i+n}{n} - \dim_K [\wp_1^{m_1-t+i} \cap \cdots \cap \wp_s^{m_s-t+i}]_i,$$
$i=0, \ldots, t$. Then by using Proposition \ref{prop44} we get
\begin{align*}
\binom{t+m}{m} - H_{R_{m}/I_{\phi_{n,m}(Z)}}(t) &= \binom{t+n}{n} - H_{R_n/(\wp_1^{m_1} \cap \cdots \cap \wp_s^{m_s})}(t) \\
&+  \underset{i=0}{\overset{t-1}{\sum}}\binom{m-n-1+t-i}{t-i} \left( \binom{i+n}{n} - H_{R_n/(\wp_1^{m_1-t+i} \cap \cdots \cap \wp_s^{m_s-t+i})}(i)\right) .
\end{align*}
It follows that 
\begin{align*}H_{R_{m}/I_{\phi_{n,m}(Z)}}(t) &= H_{R_n/I_Z}(t) + \binom{t+m}{m} - \binom{t+n}{n} \\
	&- \underset{i=0}{\overset{t-1}{\sum}}\binom{m-n-1+t-i}{t-i} \left( \binom{i+n}{n} - H_{R_n/(\wp_1^{m_1-t+i} \cap \cdots \cap \wp_s^{m_s-t+i})}(i)\right) .\end{align*}
\end{proof}

\begin{corollary}\label{cor46} Let $Z=m_1P_1+\cdots+m_sP_s$ be fat points in $\mathbb P^n$ and $\phi_{n,m}: \mathbb P^n \to \mathbb P^{m}$ be the embedding.  

a) If $m=n+1$ and $0\le t < \reg(R_n/I_Z)$, then 	
$$H_{R_{n+1}/I_{\phi_{n,n+1}(Z)}}(t) = H_{R_n/I_Z}(t) + \underset{i=0}{\overset{t-1}{\sum}} H_{R_n/(\wp_1^{m_1-t+i} \cap \cdots \cap \wp_s^{m_s-t+i})}(i).$$

b) $H_{R_{m}/I_{\phi_{n,m}(Z)}}(t) \ge  H_{R_n/I_Z}(t).$

c) If there exists $m_j \ge 2$, then 
	$$H_{R_{m}/I_{\phi_{n,m}(Z)}}(t) >  H_{R_n/I_Z}(t).$$
\end{corollary}
\begin{proof} a) If $m=n+1$ and $0\le t < \reg(R_n/I_Z)$, then by using Theorem \ref{thm45} we get
\begin{align*}H_{R_{n+1}/I_{\phi_{n,n+1}(Z)}}(t) &= H_{R_n/I_Z}(t) + \binom{t+n+1}{n+1} - \binom{t+n}{n} \\
	&- \underset{i=0}{\overset{t-1}{\sum}}\binom{t-i}{t-i} \left( \binom{i+n}{n} - H_{R_n/(\wp_1^{m_1-t+i} \cap \cdots \cap \wp_s^{m_s-t+i})}(i)\right) \\
&=H_{R_n/I_Z}(t) + \binom{t+n+1}{n+1} - \binom{t+n}{n} - \underset{i=0}{\overset{t-1}{\sum}}\binom{i+n}{n}  \\
& + \underset{i=0}{\overset{t-1}{\sum}} H_{R_n/(\wp_1^{m_1-t+i} \cap \cdots \cap \wp_s^{m_s-t+i})}(i)\\
&=H_{R_n/I_Z}(t) + \underset{i=0}{\overset{t-1}{\sum}} H_{R_n/(\wp_1^{m_1-t+i} \cap \cdots \cap \wp_s^{m_s-t+i})}(i).\\
\end{align*}
	
b) If $0\le t < \reg(R_n/I_Z)$, then by a) we get
\begin{align*}
	H_{R_{n+1}/I_{\phi_{n,n+1}(Z)}}(t) &\ge  H_{R_{n}/I_Z}(t),\\
		H_{R_{n+2}/I_{\phi_{n+1,n+2}\circ\phi_{n,n+1}(Z)}}(t) &\ge  H_{R_{n+1}/I_{\phi_{n,n+1}(Z)}}(t),\\
	\cdots\cdots\cdots\cdots\cdots\cdots\cdots\cdots\cdots\cdots&\cdots\cdot\cdots\cdots\cdots\cdots\cdots\cdots\cdots\\
			H_{R_{m}/I_{\phi_{m-1,m}\circ\phi_{m-2,m-1}\circ\cdots \circ\phi_{n+1,n+2}\circ\phi_{n,n+1}(Z)}}(t) &\ge 	H_{R_{m-1}/I_{\phi_{m-1,m-1}\circ\cdots \circ\phi_{n+1,n+2}\circ\phi_{n,n+1}(Z)}}(t).
\end{align*}
These imply that $$H_{R_{m}/I_{\phi_{m-1,m}\circ\phi_{m-2,m-1}\circ\cdots \circ\phi_{n+1,n+2}\circ\phi_{n,n+1}(Z)}}(t) \ge H_{R_{n}/I_Z}(t).$$ But  $$\phi_{n,m}(Z)=\phi_{m-1,m}\circ\phi_{m-1,m-1}\circ\cdots \circ\phi_{n+1,n+2}\circ\phi_{n,n+1}(Z).$$ Therefore,
$$H_{R_{m}/I_{\phi_{n,m}(Z)}}(t) \ge  H_{R_n/I_Z}(t).$$
	
	If $t \ge \reg(R_n/I_Z)$, then by Proposition \ref{prop41} we get 
	$$H_{R_{n+1}/I_{\phi_{n,n+1}(Z)}}(t) \ge  H_{R_n/I_Z}(t).$$

c) If there exists $m_j \ge 2$ and $0\le t < \reg(R_n/I_Z)$, we have $\wp_j^{m_j-1} \ne R_n$. So, \\ 
	$H_{R_n/(\wp_1^{m_1-1} \cap \cdots \cap \wp_s^{m_s-1})}(t-1) >0.$ By using a) we get  
	$$H_{R_{n+1}/I_{\phi_{n,n+1}(Z)}}(t) >  H_{R_n/I_Z}(t).$$
	By using b) we have 
		$$H_{R_{m}/I_{\phi_{n+1,m}\circ\phi_{n,n+1}(Z)}}(t) \ge H_{R_{n+1}/I_{\phi_{n,n+1}(Z)}}(t).$$
		But $\phi_{n,m}(Z)=\phi_{n+1,m}\circ\phi_{n,n+1}(Z)$. Therefore,
		$$H_{R_{m}/I_{\phi_{n,m}(Z)}}(t) > H_{R_{n}/I_Z}(t).$$
 If there exists $m_j \ge 2$ and $t \ge \reg(R_n/I_Z)$, then by Proposition \ref{prop41} we get	
 $$H_{R_{m}/I_{\phi_{n,m}(Z)}}(t) > H_{R_{n}/I_Z}(t).$$		
\end{proof}

\noindent Phan Van Thien, \\ Department of
Mathematics,  University of Education, Hue University, 34 Le Loi Street, Hue City, Vietnam \\ 
Email address:  pvthien@hueuni.edu.vn\\

\noindent Phan Quang Nhu Anh, \\ Faculty  of
Mathematics,  University of Science and Education - The University of Da Nang, 459 Ton Duc Thang Street, Da Nang City, Vietnam \\ 
Email address:  pqnanh@ued.udn.vn\\


\end{document}